\documentclass[twoside,leqno, british]{amsart}
\usepackage{amsmath, amssymb, amsbsy, amsthm}
\usepackage{babel,eucal,url,amssymb,
enumerate,amscd,
}

\usepackage[pagebackref]{hyperref}

\DeclareMathOperator{\const}{const}
\DeclareMathOperator{\dd}{d}

\begin{document}


\title[Quaternionic K\"ahler manifolds with Hermitian and Norden metrics]
{Quaternionic K\"ahler manifolds \\ with Hermitian and Norden
metrics}

\author{Mancho Manev}


\frenchspacing

\newcommand{\ie}{i.~e. }
\newcommand{\X}{\mathfrak{X}}
\newcommand{\R}{\mathbb{R}}
\newcommand{\W}{\mathcal{W}}
\newcommand{\K}{\mathcal{K}}
\newcommand{\s}{\mathfrak{S}}
\newcommand{\n}{\nabla}
\newcommand{\al}{\alpha}
\newcommand{\bt}{\beta}
\newcommand{\gm}{\gamma}
\newcommand{\om}{\omega}
\newcommand{\lm}{\lambda}
\newcommand{\ta}{\theta}
\newcommand{\ea}{\varepsilon_\alpha}
\newcommand{\eb}{\varepsilon_\bt}
\newcommand{\eg}{\varepsilon_\gamma}
\newcommand{\sa}{\sum_{\al=1}^3}
\newcommand{\sbt}{\sum_{\bt=1}^3}
\newcommand{\ee}{\end{equation}}
\newcommand{\be}[1]{\begin{equation}\label{#1}}
\def\bea{\begin{eqnarray}}
\def\eea{\end{eqnarray}}
\newcommand{\norm}[1]{\left\Vert#1\right\Vert ^2}
\newcommand{\nJ}[1]{\norm{\nabla J_{#1}}}

\newcommand{\thmref}[1]{The\-o\-rem~\ref{#1}}
\newcommand{\propref}[1]{Pro\-po\-si\-ti\-on~\ref{#1}}
\newcommand{\secref}[1]{\S\ref{#1}}
\newcommand{\lemref}[1]{Lem\-ma~\ref{#1}}
\newcommand{\dfnref}[1]{De\-fi\-ni\-ti\-on~\ref{#1}}
\newcommand{\corref}[1]{Corollary~\ref{#1}}

\renewcommand{\thefootnote}{\fnsymbol{footnote}}

\numberwithin{equation}{section}
\newtheorem{thm}{Theorem}[section]
\newtheorem{lem}[thm]{Lemma}
\newtheorem{prop}[thm]{Proposition}
\newtheorem{cor}[thm]{Corollary}
\newtheorem{probl}[thm]{Problem}
\newtheorem{concl}[thm]{Conclusion}
\newtheorem*{ack}{Acknowledgements}

\newtheorem{defn}{Definition}[section]
\newtheorem{rem}{Remark}[section]
\newtheorem{exa}{Example}

\hyphenation{Her-mi-ti-an ma-ni-fold ah-ler-ian}




\begin{abstract}
Almost hypercomplex manifolds with Hermitian and Norden metrics
and more specially the corresponding quaternionic K\"ahler
manifolds are considered. Some necessary and sufficient conditions
the investigated manifolds be isotropic hyper-K\"ahlerian and flat
are found. It is proved that the quaternionic K\"ahler manifolds
with the considered metric structure are Einstein for dimension at
least 8. The class of the non-hyper-K\"ahler quaternionic K\"ahler
manifold of the considered type is determined.
\end{abstract}

\maketitle

\setcounter{tocdepth}{2} \tableofcontents

\section*{Introduction}

In this work\footnote{This work was partially supported by the
Scientific Researches Fund at the University of Plovdiv.} we
continue the investigations on a manifold $M$ with an almost
hypercomplex structure $H$. We equip this almost hypercomplex
manifold $(M,H)$ with a metric structure $G$, generated by a
pseudo-Riemannian metric $g$ of neutral signature (\cite{GrMa},
\cite{GrMaDi}).

It is known, if $g$ is a Hermitian metric on $(M,H)$, the derived
metric structure $G$ is the known hyper-Hermitian structure. It
consists of the given Hermitian metric $g$ with respect to the
three almost complex structures of $H$ and the three K\"ahler
forms associated with $g$ by $H$ \cite{AlMa}.

In  our case the considered metric structure $G$ has a different
type of compatibility with $H$. The structure $G$ is generated by
a neutral metric $g$ such that the one (resp., the other two) of
the almost complex structures of $H$ acts as an isometry (resp.,
act as anti-isometries) with respect to $g$ in each tangent fibre.
Let
the almost complex structures of $H$ act as isometries or
anti-isometries with respect to the metric. Then the existence of
an anti-isometry generates exactly the existence of one more
anti-isometry and an isometry. %
Thus, $G$ contains the metric $g$ and three (0,2)-tensors
associated by $H$ -- a K\"ahler form and two metrics of the same
type.
The existence of such bilinear forms on an almost hypercomplex
manifold is proved in \cite{GrMa}.
The neutral metric $g$ is Hermitian with respect to the one almost
complex structure of $H$ and $g$ is an anti-Hermitian (\ie a
Norden) metric regarding the other two almost complex structures
of $H$. For this reason we call the derived almost hypercomplex
manifold $(M,H,G)$ an \emph{almost hypercomplex manifold with
NH-metric structure} or an \emph{almost hypercomplex NH-manifold}.

If the three almost complex structures of $H$ are parallel with
respect to the Levi-Civita connection $\nabla$ of $g$ then such
hypercomplex NH-manifolds of K\"ahler type we call
\emph{hyper-K\"ahler NH-manifolds}, which are flat according to
\cite{GrMaDi}.

In the first section we recall some facts about the almost
hypercomplex NH-manifolds known from \cite{AlMa}, \cite{GrMa},
\cite{GrMaDi}, \cite{Ma09}.

In the second section we introduce the corresponding quaternionic
K\"ahler manifold of an almost hypercomplex manifold with
NH-metric structure. We establish that the quaternionic K\"ahler
NH-manifolds are Einstein for dimension $4n\geq 8$. For
comparison, it is known that the quaternionic K\"ahler manifolds
with hyper-Hermitian metric structure are Einstein for all
dimensions.

In the third section we consider the location of the quaternionic
K\"ahler NH-manifolds in the classification of the corresponding
almost hypercomplex manifolds with respect to the covariant
derivatives of the almost complex structures. We get only one
class (except the general one) of the considered classification
where these manifolds are non-hyper-K\"ahlerian and consequently
non-flat always.

In the fourth section we characterize the obtained in the previous
chapter non-hyper-K\"ahler quaternionic K\"ahler NH-manifolds.

The basic problem of this work is the existence and the geometric
characteristics of the quaternionic K\"ahler NH-manifolds. The
main results of this paper is that every quaternionic K\"ahler
NH-manifold is Einstein for dimension at least 8 and it is not
flat hyper-K\"ahlerian only when belongs to the general class
$\W_1\oplus\W_2\oplus\W_3$ or the class $\W_1\oplus\W_3$, where
the manifold is Ricci-symmetric.

\section{Almost hypercomplex manifolds with NH-metric structure}

Let $(M,H)$ be an almost hypercomplex manifold, \ie $M$ is a
$4n$-dimension\-al differentiable manifold and $H=(J_1,J_2,J_3)$
is a triple of almost complex structures on $M$ with the
properties:
\be{J123} J_\al=J_\bt\circ J_\gm=-J_\gm\circ J_\bt, \qquad
J_\al^2=-I\ee for all cyclic permutations $(\al, \bt, \gm)$ of
$(1,2,3)$ and $I$ denotes the identity \cite{AlMa}.

Let $g$ be a neutral metric on $(M,H)$ with the properties
\be{gJJ} %
g(x,y)=\ea g(J_\al x,J_\al y), \ee %
where
\[ \ea=
\begin{cases}
\begin{array}{ll}
1, \quad & \al=1;\\[4pt]
-1, \quad & \al=2;3.
\end{array}
\end{cases}
\]
Moreover, the associated (K\"ahler) 2-form $g_1$ and the
associated neutral metrics $g_2$ and $g_3$ are determined by
\be{gJ} g_\al(x,y)=g(J_\al x,y)=-\ea g(x,J_\al y). \ee

A structure $(H,G)=(J_1,J_2,J_3,g,g_1,g_2,g_3)$ is introduced and
investigated in \cite{GrMa}, \cite{GrMaDi}, \cite{ManSek},
\cite{Ma05} and \cite{Ma09}. Now we call it an \emph{almost
hypercomplex NH-metric structure} on $M$. Then, a manifold with
such structure $(M,H,G)$ we call an \emph{almost hypercomplex
manifold with NH-metric structure} or an \emph{almost hypercomplex
NH-manifold}.

The structural tensors of a such manifold are the following three
$(0,3)$-tensors
\begin{equation}\label{F}
F_\al (x,y,z)=g\bigl( \left( \n_x J_\al
\right)y,z\bigr)=\bigl(\n_x g_\al\bigr) \left( y,z \right),
\end{equation}
where $\n$ is the Levi-Civita connection generated by $g$.
The corresponding Lie 1-forms $\ta_\al$ are defined by
\begin{equation}\label{theta-al}
\ta_\al(\cdot)=g^{ij}F_\al(e_i,e_j,\cdot)
\end{equation}%
for an arbitrary basis $\{e_1,e_2,\dots, e_{4n}\}$ of $T_pM$,
$p\in M$.

In \cite{GrMaDi} we study the so-called \emph{hyper-K\"ahler
manifolds with NH-metric structure} (or pseudo-hyper-K\"ahler
manifolds), \ie the almost hypercomplex NH-manifold in the class
$\K$, where $\n J_\al=0$ for all $\al=1,2,3$. A sufficient
condition $(M,H,G)$ be in $\K$ is this manifold be of
K\"ahler-type with respect to two of the three complex structures
of $H$ \cite{GrMa}.
%

As $g$ is an indefinite metric, there exist isotropic vectors $x$
on $M$, \ie{} \(g(x,x)=0\), \(x\neq 0\). 
In \cite{GrMa} we define the invariant square norm
\begin{equation}\label{nJ}
\nJ{\al}= g^{ij}g^{kl}g\bigl( \left( \nabla_i J_\al \right) e_k,
\left( \nabla_j J_\al \right) e_l \bigr),
\end{equation}
where $\{e_1,e_2,\dots, e_{4n}\}$ of $T_pM$, $p\in M$. We say that
an almost hypercomplex NH-manifold is an \emph{isotropic
hyper-K\"ahler NH-manifold} if $\nJ{\al}=0$ for every $J_\al$ of
$H$. In \cite{GrMa} such a manifold is called an isotropic
hyper-K\"ahler manifold. Clearly, if $(M,H,G)$ is a hyper-K\"ahler
NH-manifold, then it is an isotropic hyper-K\"ahler NH-manifold.
The inverse statement does not hold.

Let us consider the Nijenhuis tensors $N_\al$ for $J_\al$ given by
$N_\al(x,y) = \left[J_\al x,J_\al y \right]
    -J_\al\left[J_\al x,y \right]
    -J_\al\left[x,J_\al y \right]
    -\left[x,y \right]$ for $x$, $y$ $\in T_pM$.
It is well known that the almost hypercomplex structure
$H=(J_\al)$ is a {\em hypercomplex structure\/} if $N_\al$
vanishes for each $\al=1,2,3$. Moreover, it is known that one
almost hypercomplex structure $H$ is hypercomplex if and only if
two of the structures $J_\al$ $(\al=1,2,3)$ are integrable. This
means that two of the tensors $N_\al$ vanish \cite{AlMa}.

Let us note that according to \eqref{gJJ} the manifold $(M,J_1,g)$
is almost Hermitian and the manifolds $(M,J_\al,g)$, $\al=2,3$,
are almost complex manifolds with NH-metric structure (or
$B$-metric) \cite{GaBo}. The basic classes of the mentioned two
types of manifolds are given in \cite{GrHe} and \cite{GaBo},
respectively, and  they are determined for dimension $4n$ as
follows:
\begin{equation}\label{cl-H}
\begin{split}
&\W_1(J_1):\; F_1(x,y,z)=-F_1(y,x,z); \\
&\W_2(J_1):\; \mathop{\s}_{x,y,z}\bigl\{F_1(x,y,z)\bigr\}=0; \\
&\W_3(J_1):\; F_1(x,y,z)=F_1(J_1x,J_1y,z),\quad \ta_1=0; \\
&\W_4(J_1):\; F_1(x,y,z)=\frac{1}{4n-2}
                \left\{g(x,y)\ta_1(z)-g(x,z)\ta_1(y)\right. \\
&\phantom{\W_4(J_1):\; F_1(x,y,z)=\frac{1}{4n-2}
                \left\{\right.}
                \left.-g(x,J_1y)\ta_1(J_1z)+g(x,J_1z)\ta_1(J_1y)
                \right\}
\end{split}
\end{equation}
and for $\al=2$ or $3$
\begin{equation}\label{cl-N}
\begin{split}
&\W_1(J_\al):\; F_\al(x,y,z)=\frac{1}{4n}\bigl\{
g(x,y)\ta_\al(z)+g(x,z)\ta_\al(y)\bigr.\\
&\phantom{\W_1(J_\al):\; F_\al(x,y,z)=\frac{1}{4n}\bigl\{\bigr.} %
\bigl.+g(x,J_\al y)\ta_\al(J_\al z)
    +g(x,J_\al z)\ta_\al(J_\al y)\bigr\};\\
&\W_2(J_\al):\; \mathop{\s}_{x,y,z}
\bigl\{F_\al(x,y,J_\al z)\bigr\}=0,\quad \ta_\al=0;\\
&\W_3(J_\al):\; \mathop{\s}_{x,y,z} \bigl\{F_\al(x,y,z)\bigr\}=0,
\end{split}
\end{equation}
where $\s $ is the cyclic sum by three arguments $x$, $y$, $z$.

The special class $\W_0(J_\al):$ $F_\al=0$ $(\al =1,2,3)$ of the
K\"ahler-type manifolds belongs to any other class within the
corresponding classification.


Let the curvature tensor $R$ of the Levi-Civita connection
$\nabla$, generated by $g$, be defined, as usually, by
$R(x,y)z=\left[\nabla_x,\nabla_y\right] z -
\nabla_{\left[x,y\right]} z$. The corresponding $(0,4)$-tensor is
determined by $R(x,y,z,w)=g\left(R(x,y)z,w\right)$. Obviously, $R$
is a \emph{K\"ahler-type tensor} on an arbitrary hyper-K\"ahler
NH-manifold, \ie
\be{R-kel}%
\begin{array}{l}
R(x,y,z,w)=\ea R(x,y,J_\al z,J_\al w)=\ea R(J_\al x,J_\al y,z,w).
\end{array}
\ee

A basic property of the hyper-K\"ahler NH-manifolds is given in
\cite{GrMaDi} by the following
\begin{thm}[\cite{GrMaDi}]\label{thm-R=0}
Each hyper-K\"ahler NH-manifold is a flat pseudo-Rie\-mann\-ian
ma\-ni\-fold of signature $(2n,2n)$. $\hfill\Box$
\end{thm}

In \cite{Ma09} it is proved the following more general property.
\begin{thm}[\cite{Ma09}]\label{th-0}
Every K\"ahler-type tensor on an almost hypercomplex
NH-ma\-ni\-fold is ze\-ro. \hfill $\Box$
\end{thm}

\section{Quaternionic K\"ahler manifolds with NH-metric structure}

Let us consider again only an almost hypercomplex manifold
$(M,H)$. The endomorphism $Q=\lm_1 J_1+\lm_2 J_2+\lm_3 J_3$,
$\lm_i\in\R$, is called a \emph{quaternionic structure} on $(M,H)$
with an admissible basis $H$. A quaternionic structure with the
condition $\nabla Q=0$ is called a \emph{quaternionic K\"ahler
structure} on $(M,H)$. An almost hypercomplex manifold with
quaternionic K\"ahler structure is determined by
\begin{equation}\label{qK}
\left(\n_x J_\al\right)y=\om_\gm(x)J_\bt y-\om_\bt(x)J_\gm y
\end{equation}
for all cyclic permutations $(\al, \bt, \gm)$ of $(1,2,3)$, where
$\om_\al$ are local 1-forms associated to $H=(J_\al)$,
$\al=1,2,3$. \cite{AlMa}

Next, we equip the quaternionic K\"ahler manifold with an
NH-metric structure $G=(g,g_1,g_2,g_3)$, determined by \eqref{gJJ}
and \eqref{gJ}, and obtain a \emph{quaternionic K\"ahler manifold
with NH-metric structure} or a \emph{quaternionic K\"ahler
NH-manifold}.

Having in mind \eqref{qK} and \eqref{nJ}, for a quaternionic
K\"ahler NH-manifold  we obtain the following form of the square
norm of $\n J_\al$:
\begin{equation}\label{nJ-qK}
    \nJ{\al}=4n\left\{\eb\om_\gm(\Omega_\gm)+\eg\om_\bt(\Omega_\bt)\right\},
\end{equation}
where $\Omega_1$, $\Omega_2$, $\Omega_3$ are the corresponding
vectors of $\om_1$, $\om_2$, $\om_3$  regarding $g$, respectively.
Therefore we have immediately the following
\begin{prop}\label{prop-iqK}
A quaternionic K\"ahler NH-manifold is an isotropic hyper-K\"ahler
NH-manifold iff the corresponding vectors of the 1-forms $\om_1$,
$\om_2$ and $\om_3$ with respect to $g$ are isotropic vectors
regarding $g$.\hfill $\Box$
\end{prop}


Having in mind \eqref{qK}, we obtain the following property for
all cyclic permutations $(\al,\bt,\gm)$ of $(1,2,3)$:
\begin{equation}\label{RJ}
      R(x,y)J_{\al} z=J_{\al} R(x,y)z-\eta_{\bt} (x,y)J_{\gm} z+\eta_{\gm} (x,y)J_\bt z
\end{equation}
where
\begin{equation}\label{eta}
\eta_{\bt}(x,y)=\dd\om_{\bt} (x,y)+\om_{\gm} (x)\om_{\al}
(y)-\om_{\al} (x)\om_{\gm} (y)
\end{equation}
are 2-forms associated to the local 1-forms $\om_1$, $\om_2$,
$\om_3$ and therefore
\begin{equation}\label{RJJ}
    R(x,y,z,w)-\ea
    R(x,y,J_{\al}z,J_{\al}w)=\eta_{\bt}(x,y)g_{\bt}(z,w)+\eta_{\gm}(x,y)g_{\gm}(z,w).
\end{equation}
According to the antisymmetry of $R$ by the third and the forth
entries, we establish that $\eta_2=\eta_3=0$, \ie
\begin{lem}
The local 1-forms $\om_1$, $\om_2$ and  $\om_3$, determining a
quaternionic K\"ahler NH-manifold, satisfy the following
identities
\begin{equation}\label{d-om}
\begin{array}{l}
    \dd\om_2(x,y)=-\om_{3} (x)\om_{1}(y)+\om_{1} (x)\om_{3} (y),\\[4pt]
    \dd\om_3(x,y)=-\om_{1} (x)\om_{2}(y)+\om_{2} (x)\om_{1} (y).\\[4pt]
\end{array}
\end{equation}\hfill $\Box$
\end{lem}

Then, according to \eqref{d-om}, equations \eqref{RJJ} take the
form
\begin{gather}
    R(x,y,J_{1}z,J_{1}w)=R(x,y,z,w),\label{RJJ1}\\[4pt]
    R(x,y,J_{2}z,J_{2}w)=R(x,y,J_{3}z,J_{3}w)
    =-R(x,y,z,w)+\eta_{1}(x,y)g_{1}(z,w).\label{RJJ23}
\end{gather}

Having in mind \eqref{RJJ1}, \eqref{RJJ23} and \eqref{R-kel}, we
have immediately
\begin{lem}\label{lem-Rkel}
The curvature tensor $R$ of a quaternionic K\"ahler NH-manifold is
of K\"ahler-type iff $\eta_1=0$, \ie the following condition is
valid
\begin{equation}\label{om1}
\dd\om_1(x,y)=-\om_{2} (x)\om_{3} (y)+\om_{3} (x)\om_{2} (y).
\end{equation}\hfill $\Box$
\end{lem}
According to \lemref{lem-Rkel} and \thmref{thm-R=0}, we have
\begin{prop}\label{prop-R=0}
The necessary and sufficient condition an arbitrary quaternionic
K\"ahler NH-manifold  be flat is condition \eqref{om1}.\hfill
$\Box$
\end{prop}

\begin{lem}\label{l-rho-eta}
The Ricci tensor and the 2-form $\eta_1$, defined by \eqref{eta},
have the following relation on any quaternionic K\"ahler
NH-manifold:
\begin{equation}\label{rho-eta}
    \rho(x,y)=n\eta_1(J_1x, y).
\end{equation}
\end{lem}
\begin{proof}
From \eqref{RJJ23} for $z \rightarrow e_i$, $w \rightarrow J_1e_j$
by contraction with $g^{ij}$ we have
\begin{equation}\label{R123a}
\begin{array}{l}
    -g^{ij}R(x,y,J_{2}e_i,J_{3}e_j)=g^{ij}R(x,y,J_{3}e_i,J_{2}e_j)\medskip\\
\phantom{-g^{ij}R(x,y,J_{2}e_i,J_{3}e_j)}
    =-g^{ij}R(x,y,e_i,J_1e_j)+4n\eta_{1}(x,y).
\end{array}
\end{equation}
Having in mind the antisymmetry on the second pair arguments of
$R$ and $J_1=J_2J_3$, we get
\begin{equation*}
\begin{array}{l}
    -g^{ij}R(x,y,J_{2}e_i,J_{3}e_j)=g^{ij}R(x,y,J_{3}e_i,J_{2}e_j)
    =g^{ij}R(x,y,e_i,J_1e_j).
\end{array}
\end{equation*}
and therefore from \eqref{R123a} we have
\begin{equation}\label{a-eta}
    g^{ij}R(x,y,e_i,J_1e_j)=2n\eta_1(x,y).
\end{equation}

After that, from \eqref{a-eta}, applying the properties of the
curvature tensor $R$, \eqref{gJJ} for $\al=1$ and \eqref{RJJ1}, we
obtain consequently
\[
\begin{array}{l}
2n\eta_1(x,y)=g^{ij}R(x,y,e_i,J_1e_j)=g^{ij}\{-R(x,e_i,J_1e_j,y)-R(x,J_1e_j,y,e_i)\}\medskip\\
=g^{ij}R(x,e_i,y,J_1e_j)+g^{ij}R(x,e_j,y,J_1e_i)=2g^{ij}R(x,e_i,y,J_1e_j)\medskip\\
=-2g^{ij}R(e_i,x,y,J_1e_j)=2g^{ij}R(e_i,x,J_1y,e_j)=2\rho(x,J_1y),\medskip\\
\end{array}
\]
\ie
\begin{equation}\label{eta-rho}
    \eta_1(x,y)=\frac{1}{n}\rho(x,J_1y).
\end{equation}
Because of the symmetry of $\rho$ and the antisymmetry of $\eta_1$
we have the property
\begin{equation}\label{etaJ1}
\eta_1(x,J_1y)=-\eta_1(J_1x,y)
\end{equation}
and therefore
\begin{equation}\label{rJJ1}
    \rho(J_1x,J_1y)=\rho(x,y).
\end{equation}
Hence, from \eqref{eta-rho}, \eqref{etaJ1} and \eqref{rJJ1}, we
obtain \eqref{rho-eta}.
\end{proof}

\begin{prop}\label{prop-r=0}
    A quaternionic K\"ahler NH-manifold  is
    Ricci-flat iff it is flat.
\end{prop}
\begin{proof}
Using \lemref{l-rho-eta}, property \eqref{RJJ23} takes the form
\begin{equation}\label{RJJ23r}
\begin{array}{l}
    R(x,y,J_{2}z,J_{2}w)=R(x,y,J_{3}z,J_{3}w)\medskip\\
    \phantom{R(x,y,J_{2}z,J_{2}w)}
    =-R(x,y,z,w)-\frac{1}{n}\rho(J_1x,y)g(J_{1}z,w).
\end{array}
\end{equation}
Then, according to \eqref{RJJ23r}, \eqref{R-kel} and
\thmref{th-0}, we obtain the equivalence in the statement.
\end{proof}

\begin{thm}\label{thm-Ein}
Quaternionic K\"ahler manifolds with NH-metric structure are
Einstein for dimension $4n\geq 8$.
\end{thm}
\begin{proof}
By virtue of \eqref{RJJ1}, \eqref{RJJ23r}  and \eqref{rJJ1} we
obtain the following properties
\begin{gather}\label{RJJJJ1}
    R(J_{1}x,J_{1}y,J_{1}z,J_{1}w)=R(x,y,z,w),\medskip\\
\begin{array}{l}\label{RJJJJ23}
    R(J_{2}x,J_{2}y,J_{2}z,J_{2}w)=R(J_{3}x,J_{3}y,J_{3}z,J_{3}w)\medskip\\
    =R(x,y,z,w)-\frac{1}{n}g(x,J_1y)\rho(J_1z,w)+\frac{1}{n}\rho(J_2x,J_3y)g(J_1z,w).
\end{array}
\end{gather}

Hence, for the Ricci tensor we have \eqref{rJJ1} and
\begin{equation}\label{rho4J23}
    (n^2-1)\rho(J_{2}y,J_{2}z)=(n^2-1)\rho(J_{3}y,J_{3}z)
    =-(n^2-1)\rho(y,z).
\end{equation}
Then for $n>1$ the Ricci tensor is hybrid with respect to $J_2$
and $J_3$, \ie
\begin{equation*}\label{rhoJ2J3}
    \rho(J_{2}y,J_{2}z)=\rho(J_{3}y,J_{3}z)
    =-\rho(y,z).
\end{equation*}

The conditions \eqref{RJJ1}, \eqref{RJJ23} and \eqref{rho-eta}
imply for $n>1$ the following
\begin{equation*}\label{4Rrho}
    A(x,z)=-\frac{2}{n}\rho(x,x)g(z,z)=-\frac{2}{n}g(x,x)\rho(z,z),
\end{equation*}
where
\begin{equation*}\label{A}
\begin{array}{l}
A(x,z)=R(x,J_{1}x,z,J_{1}z)-R(x,J_{1}x,J_{2}z,J_{3}z)\\[4pt]
    \phantom{A(x,z)=}
        -R(J_{2}x,J_{3}x,z,J_{1}z)+R(J_{2}x,J_{3}x,J_{2}z,J_{3}z)\\[4pt]
\end{array}
\end{equation*}

Then for arbitrary non-isotropic vectors we have $\rho=\lm g$,
$\lm\in\R$.
\end{proof}

By \thmref{thm-Ein}, identity \eqref{RJJ23r} implies the following
corollary for $n\geq 2$:
\begin{equation}\label{RJJ23t}
\begin{split}
    &R(x,y,J_{2}z,J_{2}w)=R(x,y,J_{3}z,J_{3}w)\medskip\\
    &\phantom{R(x,y,J_{2}z,J_{2}w)}
    =-R(x,y,z,w)-\frac{\tau}{4n^2}g(J_1x,y)g(J_{1}z,w).
\end{split}
\end{equation}

Therefore from \eqref{RJJ23t}, using \eqref{R-kel} and
\thmref{th-0}, we obtain the following
\begin{prop}\label{prop-t=0}
    A quaternionic K\"ahler NH-manifold  of dimension $4n\geq 8$ is
    scalar flat iff it is flat. \hfill$\Box$
\end{prop}

\begin{prop}\label{prop-d-om}
    A quaternionic K\"ahler NH-manifold  of dimension $4n\geq 8$
    is determined by the local 1-forms satisfying the conditions
    \eqref{d-om} and
\begin{equation*}\label{om1=}
\dd\om_1(x,y)=-\om_{2} (x)\om_{3} (y)+\om_{3} (x)\om_{2}
(y)-\frac{\tau}{4n^2}g(J_1x,y).
\end{equation*} \hfill$\Box$
\end{prop}

\section{Quaternionic K\"ahler NH-manifolds in a classification of almost hypercomplex
NH-manifolds}

Firstly, let us consider the case when $H$ is (integrable)
hypercomplex structure, \ie when $N_\al$ vanishes for each
$\al=1,2,3$.

 The Nijenhuis tensor and its associated tensor for
each $J_\al$ are determined as follows:
\begin{equation}\label{NJ}
    \begin{array}{l}
      N_\al(x,y)=\left(\n_x J_\al\right)J_\al y-\left(\n_y
J_\al\right)J_\al x+\left(\n_{J_\al x}
J_\al\right)y-\left(\n_{J_\al y} J_\al\right)x, \\[4pt]
      N_\al^*(x,y)=\left(\n_x J_\al\right)J_\al y+\left(\n_y
J_\al\right)J_\al x+\left(\n_{J_\al x}
J_\al\right)y+\left(\n_{J_\al y} J_\al\right)x. \\
    \end{array}
\end{equation}
Therefore, according to \eqref{qK}, for the quaternionic K\"ahler
 manifolds we have
\begin{equation}\label{NN*}
\begin{array}{l}
N_\al(x,y)=-\left[\om_\gm(x)+\om_\bt(J_\al x)\right]J_\gm y
-\left[\om_\bt(x)-\om_\gm(J_\al x)\right]J_\bt y\\[4pt]
\phantom{N_\al(x,y)=} +\left[\om_\gm(y)+\om_\bt(J_\al
y)\right]J_\gm x +\left[\om_\bt(y)-\om_\gm(J_\al y)\right]J_\bt x,
\\[4pt]
N_\al^*(x,y)=-\left[\om_\gm(x)+\om_\bt(J_\al x)\right]J_\gm y
-\left[\om_\bt(x)-\om_\gm(J_\al x)\right]J_\bt y\\[4pt]
\phantom{N_\al^*(x,y)=} -\left[\om_\gm(y)+\om_\bt(J_\al
y)\right]J_\gm x -\left[\om_\bt(y)-\om_\gm(J_\al y)\right]J_\bt x.
\end{array}
\end{equation}
The last equations imply immediately the next two lemmas.
\begin{lem}\label{lm-N}
The tensors $N_\al$ and $N_\al^*$ vanish iff
$\om_\gm=-\om_\bt\circ J_\al$ for any fixed cyclic permutation
$(\al, \bt, \gm)$ of $(1,2,3)$.\hfill $\Box$
\end{lem}
\begin{lem}\label{lm-NN}
The tensors $N_\al$ and $N_\al^*$ ($\al=1,2,3$) vanish iff
\begin{equation}\label{qK-usl}
\om_\al=\om_\bt\circ J_\gm=-\om_\gm\circ J_\bt
\end{equation}
for  cyclic permutations $(\al, \bt, \gm)$ of $(1,2,3)$.\hfill
$\Box$
\end{lem}

Now, according to \eqref{qK} and \eqref{gJ}, the structural
tensors and their corresponding Lie 1-forms of the derived
quaternionic K\"ahler NH-manifold, defined by \eqref{F}
 and \eqref{theta-al},  have the form
\begin{gather}
    F_\al(x,y,z)=\om_\gm(x)g(J_\bt
y,z)-\om_\bt(x)g(J_\gm y,z),\label{qK-F}\\[4pt]
    \ta_\al(z)=-\eb\om_\gm(J_\bt z)+\eg\om_\bt(J_\gm z).\label{qK-theta}
\end{gather}

\begin{prop}\label{prop-int}
If a quaternionic K\"ahler NH-manifold $(M,H,G)$ is integrable,
then it is a hyper-K\"ahler NH-manifold, \ie
\[
(M,H,G)\in\left(\W_3\oplus\W_4\right)(J_1)\cap\left(\W_1\oplus\W_2\right)(J_2)
\cap\left(\W_1\oplus\W_2\right)(J_3)
\]
\[
 \Rightarrow \quad
(M,H,G)\in\K.
\]
\end{prop}
\begin{proof}
Let $(M,H,G)$ be an integrable hypercomplex NH-manifold, \ie
$(M,H,G)$ belongs to the class $\W_3\oplus\W_4$ with respect to
$J_1$ in \eqref{cl-H} and $(M,H,G)$ is an element of
 $\W_1\oplus\W_2$ regarding $J_2$ and $J_3$, according to
\eqref{cl-N}.

Therefore $N_1=N_2=N_3=0$ hold and then, according to
\lemref{lm-NN}, conditions \eqref{qK-usl} are valid. Hence,
according to $\ea+\eb+\eg=-1$ and $\ea\eb\eg=1$, relation
\eqref{qK-theta} takes the form
\[
\ta_\al=-(1+\ea)\om_\al,
\]
which imply
\[
\ta_1=-2\om_1,\qquad \ta_2=\ta_3=0.
\]

On the other hand, $N_\al=N_\al^*=0$  and \eqref{NJ}  imply
$\left(\n_x J_\al\right)y=\left(\n_{J_\al x} J_\al\right)J_\al y$
and finally the fact that the manifold is hyper-K\"ahlerian with
an NH-metric structure.
\end{proof}
\begin{prop}\label{prop-23}
If an almost hypercomplex NH-manifold $(M,H,G)$, determined by the
properties $\ta_2=\ta_3=0$, is quaternionic K\"ahlerian, then it
is a hyper-K\"ahler NH-manifold, \ie
\[
(M,H,G)\in%
\left(\W_2\oplus\W_3\right)(J_2)
\cap\left(\W_2\oplus\W_3\right)(J_3)\quad \Rightarrow \quad
(M,H,G)\in\K.
\]
\end{prop}
\begin{proof}
Since $\ta_2=\ta_3=0$, we have $N_2=N_3=0$, because of
\eqref{qK-theta} and \eqref{NN*}. Consequently, $N_1$ vanishes,
too. Then, according to \propref{prop-int} and conditions
\eqref{cl-N}, the considered manifold belongs to the class $\K$.
\end{proof}

Let us remark, using \eqref{cl-N}, that an almost complex manifold
with Norden metric belongs to $\W_1\oplus\W_3$ regarding $J_\al$
iff the following property holds for $\al=2$ or $3$, respectively
\begin{equation}\label{W1W3}
\mathop{\s}\limits_{x,y,z}F_\al(x,y,z)=\frac{1}{2n}\mathop{\s}\limits_{x,y,z}\left\{
g(x,y)\ta_\al(z)+g(J_\al x,y)\ta_\al(J_\al z)\right\}.
\end{equation}

\begin{prop}\label{prop-13}
Let $(M,H,G)$ be an almost hypercomplex NH-manifold belonging to
the class $\W_1\oplus\W_3$ with respect to $J_2$ and $J_3$. If
$(M,H,G)$  is quaternionic K\"ahlerian, then it is a K\"ahler
manifold with
respect to $J_1$, \ie %
\[
(M,H,G)\in\left(\W_1\oplus\W_3\right)(J_2)\cap\left(\W_1\oplus\W_3\right)(J_3)\quad
\Rightarrow\quad (M,H,G)\in\W_0(J_1).
\]
Moreover, we have
\begin{equation}\label{F123}
    \begin{array}{ll}
      \left(\n_xJ_1\right)y=0, \\[4pt]
      \left(\n_xJ_2\right)y=\om_1(x)J_3 y, \qquad & \om_1(x)=-\ta_2(J_3 x),\\[4pt]
      \left(\n_xJ_3\right)y=-\om_1(x)J_2 y, \qquad & \om_1(x)=\ta_3(J_2 x).\\[4pt]
    \end{array}
\end{equation}
\end{prop}
\begin{proof}
From \eqref{W1W3} for $\al=2$ and $3$ we obtain
\[
\ta_2=\om_1\circ J_3,\qquad \ta_3=-\om_1\circ J_2.
\]
Then, according to \eqref{qK-theta}, we get
\begin{equation}\label{om23}
\om_2=\om_3=0
\end{equation}
and therefore we obtain \eqref{F123}.
\end{proof}

From \propref{prop-23} and \propref{prop-13} we have directly
\begin{cor}\label{cor-W3}
Let $(M,H,G)$ be an almost hypercomplex NH-manifold, belonging to
the class $\W_3$ with respect to $J_2$ and $J_3$. If $(M,H,G)$ is
quaternionic K\"ahlerian, then it is a hyper-K\"ahler NH-manifold,
\ie
\[
(M,H,G)\in\W_3(J_2)\cap\W_3(J_3)\quad \Rightarrow\quad
(M,H,G)\in\K.
\]\hfill $\Box$
\end{cor}

Having in mind Propositions \ref{prop-int}--\ref{prop-13},
\corref{cor-W3} and \thmref{thm-R=0}, we give the following
\begin{concl}
Let a quaternionic K\"ahler NH-manifold $(M,H,G)$ be in some of
the classes $\W_1\oplus\W_2$ (and in particular $\W_0$, $\W_1$ and
$\W_2$), $\W_2\oplus\W_3$ and $\W_3$ with respect to both of the
structures $J_2$ and $J_3$. Then $(M,H,G)$ is a flat
hyper-K\"ahler NH-manifold. The unique class in \eqref{cl-N} with
some condition for $\n J_2$ and $\n J_3$, where $(M,H,G)$ is not
flat hyper-K\"ahlerian, is $\W_1\oplus\W_3$ and their manifolds
are determined by \eqref{F123}. \hfill $\Box$
\end{concl}

\section{Non-hyper-K\"ahler quaternionic K\"ahler NH-manifolds}

In this section we will characterize the manifold satisfying the
conditions of \propref{prop-13}. It is a non-hyper-K\"ahler
quaternionic K\"ahler NH-manifold.

We apply \eqref{om23} to \eqref{nJ-qK} and obtain the square norms
of the non-zero quantities $\n J_2$ and $\n J_3$ in the considered
case as follows:
\[
\nJ{2}=\nJ{3}=-4n\om_1(\Omega_1),
\]
where $\Omega_1$ is the corresponding vector to $\omega_1$ with
respect to $g$.
\begin{cor}
Let $(M,H,G)$ be a quaternionic K\"ahler NH-manifold, determined
by a local 1-form $\om_1$ in \eqref{F123}. It is an isotropic
hyper-K\"ahler NH-manifold iff the corresponding vector $\Omega_1$
to $\omega_1$ with respect to $g$ is an isotropic vector regarding
$g$.\hfill $\Box$
\end{cor}

Using \eqref{om23} and \propref{prop-d-om}, for the considered
manifolds here we have
\begin{prop}\label{prop-d-om=}
    Let $\om_1$ be the local 1-form of a quaternionic K\"ahler NH-manifold  of dimension $4n\geq 8$,
    determined by \eqref{F123}. Then $\om_1$ satisfies the condition
\begin{equation}\label{om1==}
\dd\om_1(x,y)=-\frac{\tau}{4n^2}g_1(x,y).
\end{equation} \hfill$\Box$
\end{prop}

According to \eqref{F123} we have $F_1(x,y,z)=\bigl(\n_x
g_\al\bigr) \left( y,z \right)=0$ and then we obtain $\dd
g_1(x,y,z)=\mathop{\s}_{x,y,z}\bigl\{F_1(x,y,z)\bigr\}=0$. Hence
and \eqref{om1==} we establish that $\tau=\const$, \ie $(M,H,G)$
determined by \eqref{F123} has a constant scalar curvature. Then,
having in mind \thmref{thm-Ein}, we get the following
\begin{prop}\label{prop-Ric-sym}
    Quaternionic K\"ahler NH-manifolds
    determined by \eqref{F123} for dimension $4n\geq 8$ are Ricci-symmetric, \ie $\n\rho=0$. \hfill$\Box$
\end{prop}

As in \propref{prop-R=0} for an arbitrary quaternionic K\"ahler
NH-manifold, in the following proposition we give a necessary and
sufficient condition the considered manifold in this section be
flat.

\begin{prop}\label{prop-om1}
Let $(M,H,G)$ be a quaternionic K\"ahler NH-manifold, determined
by a non-zero local 1-form $\om_1$ in \eqref{F123}. Then $(M,H,G)$
is flat non-hyper-K\"ahlerian iff $\om_1$ is closed.
\end{prop}
\begin{proof}
Since $M$ is a K\"ahler manifold with respect to $J_1$, then $R$
is K\"ahlerian with respect to $J_1$.

Having in mind \eqref{om23} and \eqref{eta}, we have that
$\eta_1=\dd\om_1$ in the considered case. Then identity
\eqref{RJJ23} takes the following form
\begin{equation*}\label{RR}
\begin{array}{l}
    R(x,y,J_2 z,J_2 w)
    =R(x,y,J_3 z,J_3 w)\\[4pt]
    \phantom{R(x,y,J_2 z,J_2 w)}=-R(x,y,z,w)+d\om_1(x,y)g(J_1
    z,w).
\end{array}
\end{equation*}
It is clear that $R$ is a K\"ahler-type tensor with respect to
$H=(J_\al)$ iff $\om_1$ is closed. Hence, according to
\thmref{th-0}, we obtain the statement.
\end{proof}

\begin{cor}
Let $(M,H,G)$ be a quaternionic K\"ahler NH-manifold, determined
by a non-zero local 1-form $\om_1$ in \eqref{F123}. Then $(M,H,G)$
is flat non-hyper-K\"ahlerian  iff the following identity is valid
\[
d\ta_\al(x,y)+d\ta_\al(J_1 x,J_1 y)-d\ta_\al(J_2 x,J_2
y)-d\ta_\al(J_3 x,J_3 y)=0
\]
for $\al=2$ or $\al=3$.
\end{cor}
\begin{proof}
It follows directly from \propref{prop-om1} and the relations
$\om_1=-\ta_2\circ J_3=\ta_3\circ J_2$ in \eqref{F123}.
\end{proof}


\bigskip

\noindent
\textit{
University of Plovdiv\\
Faculty of Mathematics and Informatics
\\
Department of Geometry\\
236 Bulgaria Blvd\\
Plovdiv 4003\\
Bulgaria}
\\
\texttt{e-mail: mmanev@uni-plovdiv.bg\\
http://fmi.uni-plovdiv.bg/manev}

\end{document}